\documentclass{amsart}
\usepackage{amsmath,amsthm,amsfonts, amssymb,amscd}
\usepackage[all]{xy}

\newtheorem{theorem}{Theorem}[section]
\newtheorem{lemma}[theorem]{Lemma}%[section]
\newtheorem{example}[theorem]{Example}%[section]
\newtheorem{remark}[theorem]{Remark}%[section]
\newtheorem{proposition}[theorem]{Proposition}

\newcommand{\minusre}{\hspace{0.3em}\raisebox{0.3ex}{\sl \tiny /}\hspace{0.3em}}
\newcommand{\minusli}{\hspace{0.3em}\raisebox{0.3ex}{\sl \tiny $\setminus $}\hspace{0.3em}}
\newcommand{\lex}{\,\overrightarrow{\times}\,}

\newcommand{\Ker}{\mbox{\rm Ker}}

\newcommand{\Infinit}{\mbox{\rm Infinit}}

\newcommand{\RDP}{\mbox{\rm RDP}}
\newcommand{\RIP}{\mbox{\rm RIP}}
\begin{document}
\title[On a New Construction of Pseudo Effect Algebras]{On a New Construction of Pseudo Effect Algebras}
\author[Anatolij Dvure\v{c}enskij]{Anatolij Dvure\v censkij$^{1,2}$}
\date{}%
\maketitle
\begin{center}  \footnote{


This work was supported by  the Slovak Research and Development Agency under contract APVV-0178-11,  grant VEGA No. 2/0059/12 SAV, and
CZ.1.07/2.3.00/20.0051.
 }
Mathematical Institute,  Slovak Academy of Sciences,\\
\v Stef\'anikova 49, SK-814 73 Bratislava, Slovakia\\
$^2$ Depart. Algebra  Geom.,  Palack\'{y} University\\
17. listopadu 12, CZ-771 46 Olomouc, Czech Republic\\

E-mail: {\tt dvurecen@mat.savba.sk}
\end{center}

\begin{abstract}
We define a new class of pseudo effect algebras, called kite pseudo effect algebras, which is connected not necessarily with partially ordered groups, but rather with generalized pseudo effect algebras where the greatest element is not guaranteed. Starting even with a  commutative generalized pseudo effect algebra, we can obtain a non-commutative pseudo effect algebra. We show how such kite pseudo effect algebras are tied with different types of the Riesz Decomposition Properties. We  find conditions when kite pseudo effect algebras have the least non-trivial normal ideal.
\end{abstract}

\section{Introduction}%1

A basic algebraic structure describing events observing during the measuring process in quantum mechanics is an effect algebra introduced in \cite{FoBe}. Such algebras are partial algebras with a primary notion $+$ which means that $a+b$ denotes the disjunction of mutually excluding events $a$ and $b$. This class was inspired by an algebraic counterpart of so-called POV-measures (positive operator-valued measures). The orthodox example of effect algebras is the class $\mathcal E(H)$ of Hermitian operators between the zero, $O$, and identity operator, $I$, acting on a Hilbert space $H$. Effect algebras are intensively studied during the last 20 years because they generalize Boolean algebras, orthomodular lattices and posets, and orthoalgebras.

In many important examples, an effect algebra is an interval $[0,u]$ in the positive cone of an Abelian partially ordered group (=po-group). This is true e.g. if (1) $\mathcal B(H)$ is the system of Hermitian operators of a Hilbert space $H$, then $\mathcal E(H)$ is the interval $[O,I]$ in $\mathcal B(H)$, or (2) if the effect algebra satisfies the Riesz Decomposition Property (RDP for short), c.f. \cite{Rav}.

A more general structure with a partially defined operation $+$ is a generalized effect algebra where the existence of the greatest element $1$ is not a priori guaranteed. Such an example is the set of all positive Hermitian operators $\mathcal B(H)^+$ of a Hilbert space $H$. Every generalized effect $E$ algebra can be embedded into the unitization of $E$, i.e. into the effect algebra $E\uplus \overline E,$ where $E$ is a lower part and $\overline E$ is a copy of $E$ with the reverse order in the upper part, and $\uplus$ denotes the ordinal sum.

In the Nineties, the assumption that addition $+$ is commutative was canceled in \cite{DvVe1, DvVe2}, and pseudo effect algebras were introduced as a non-commutative generalization of effect algebras. They have been published in physical journals. Some physical motivation for pseudo effect algebras with possible physical situations in quantum mechanics were presented in \cite{DvVe5}. In some important examples they are also intervals in po-groups  that are not necessarily commutative. Also in physics, there are important structures with non-commutative operations as for example, multiplication of matrices is a non-commutative operation and matrices are frequently used in mathematical physics. In particular, the class of quadratic matrices of the form $$ A(a,b)=
\left( \begin{array}{cc}
a & b \\
0& 1
\end{array}
\right)
$$
for $a>0$, $b \in (-\infty,\infty)$ with usual multiplication of matrices is a non-commutative linearly ordered group with the neutral element $A(1,0)$ and with the positive cone consisting of matrices $A(a,b)$ with $a>1$ or $a=1$ and $b\ge0$.

Similarly, generalized pseudo effect algebras were introduced in \cite{DvVe3, DvVe4} as partial algebras with partial addition $+$ where the top element $1$ can fail. If a pseudo effect algebra satisfies a stronger type of the Riesz Decomposition Property, RDP$_1$, then it is again an interval in a po-group (not necessarily Abelian) with strong unit, \cite{DvVe1, DvVe2}. If a generalized pseudo effect algebra  satisfies also RDP$_1,$ by \cite{DvVe4}, it can be embedded into the positive cone of a po-group with RDP$_1$. We note that RDP$_1$ for effect algebras coincides with RDP, but for pseudo effect algebras they can be different.

Recently in \cite{DvuK}, there was introduced a new construction of pseudo effect algebras starting from a partially ordered group $G$ with the positive cone and negative cone $G^+$ and $G^-$, respectively, using an index set $I$ and two bijective mapping $\lambda,\rho: I \to I.$ The universes  of these algebras are of the form $(G^+)^I$ down and $(G^-)^I$ up and the addition between two sequences $\langle x_i\colon i \in I\rangle$ and $\langle y_j\colon j \in I\rangle$ is ruled by properties of $\lambda$ and $\rho$. Therefore, even starting with a commutative po-group, e.g. with the group $\mathbb Z$ of integers, the resulting algebra is not necessarily commutative. The basic properties of such pseudo effect algebras, called kite pseudo effect algebras, are presented in \cite{DvuK, DvHo}.

In the present paper we generalize the construction of pseudo effect algebras when we change $G^+$ to a generalized pseudo effect algebra %(we note that generalized  pseudo effect algebras have not guaranteed the existence of the top element $1$ whereas pseudo effect algebra do),
because $G^+$ is an example of generalized pseudo effect algebras. It is interesting that the original construction works also for this case with necessary specifications, therefore, some present proofs resemble original ones from \cite{DvuK,DvHo}. If $I$ is a singleton, then our construction corresponds to the so-called unitization studied e.g. in \cite{XLGRD}, when $E$ is used as a lower part and $\overline E$ (a copy of $E$ with the reverse order as original one in $E$) is an upper part of the pseudo effect algebra $E \uplus \overline E$.

The main aim of the paper is to introduce this new construction of pseudo effect algebras because every theory is as so good as it contains the largest reservoir of important examples. Similarly as in \cite{DvuK, DvHo}, we present the basic construction with the fundamental properties of kite pseudo effect algebras. We concentrate to a relation between the Riesz Decomposition Properties of  the original generalized pseudo effect algebra and of the resulting kite pseudo effect algebra. Finally, we show also cases when the kite pseudo effect algebra with RDP$_1$ is subdirectly irreducible (equivalently, it contains the smallest non-trivial normal ideal).

The paper is organized as follows. In Section 1 we gather  basic definitions and notions from the theory of pseudo effect algebras and generalized pseudo effect algebras, theory of partially ordered groups. Section 3 introduces the construction of kite pseudo effect algebras starting with a generalized pseudo effect algebra. In addition, it studies also different types of the Riesz Decomposition Property. Section 4 concentrates to a description of subdirect irreducible kite pseudo effect algebras using the subdirect irreducibility of the original generalized pseudo effect algebra.

\section{Basic Definitions and Properties}%2

In the present section, we gather the necessary notions from theory of pseudo effect algebra and partially ordered groups which we will use in the paper.

According to \cite{DvVe1, DvVe2}, we say that a {\it pseudo effect algebra} is  a partial algebra  $ E=(E; +, 0, 1)$, where $+$ is a partial binary operation and $0$ and $1$ are constants, such that for all $a, b, c
\in E$, the following holds

\begin{enumerate}
\item[(i)] $a+b$ and $(a+b)+c$ exist if and only if $b+c$ and
$a+(b+c)$ exist, and in this case $(a+b)+c = a+(b+c)$;

\item[(ii)]
  there is exactly one $d \in E$ and
exactly one $e \in E$ such that $a+d = e+a = 1$;

\item[(iii)]
 if $a+b$ exists, there are elements $d, e
\in E$ such that $a+b = d+a = b+e$;

\item[(iv)] if $1+a$ or $a+1$ exists, then $a = 0$.
\end{enumerate}

If we define $a \le b$ if and only if there exists an element $c\in
E$ such that $a+c =b,$ then $\le$ is a partial ordering on $E$ such
that $0 \le a \le 1$ for all $a \in E.$ It is possible to show that
$a \le b$ if and only if $b = a+c = d+a$ for some $c,d \in E$. We
write $c = a \minusre b$ and $d = b \minusli a.$ Then

$$ (b \minusli a) + a = b= a + (a \minusre b),
$$
and we write $a^- = 1 \minusli a$ and $a^\sim = a\minusre 1$ for all
$a \in E.$ Then $a^-+a=1=a+a^\sim$ and $a^{-\sim}=a=a^{\sim-}$ for all $a\in E.$

We recall that a {\it po-group} (= partially ordered group) is a
group $G=(G;+,-,0)$ endowed with a partial order $\le$ such that if $a\le b,$ $a,b
\in G,$ then $x+a+y \le x+b+y$ for all $x,y \in G.$  We denote by
$G^+:=\{g \in G: g \ge 0\}$ the {\it positive cone} of $G.$ If, in addition, $G$ is a lattice under $\le$, we call it an $\ell$-group (= lattice
ordered group). An element $u \in G^+$ is said to be a {\it strong unit} (or an {\it order unit}) if, given $g \in G,$ there is an integer $n \ge 1$ such that $g \le nu.$ The pair $(G,u),$ where $u$ is a fixed strong unit of $G,$ is said to be a {\it unital po-group}. We recall that  the {\it lexicographic product} of two po-groups $G_1$ and $G_2$ is the group $G_1\times G_2,$ where the group operations are defined by coordinates, and the ordering $\le $ on $G_1 \times G_2$ is defined as follows: For $(g_1,h_1),(g_2,h_2) \in G_1 \times G_2,$  we have $(g_1,h_1)\le (g_2,h_2)$  whenever $g_1 <g_2$ or $g_1=g_2$ and $h_1\le h_2.$ For more information on po-group, see \cite{Dar,Fuc}.

We denote by  $\mathbb Z$ the commutative $\ell$-group of integers.

Let  $G$ be a po-group and fix an element $u \in G^+.$ If we set $\Gamma(G,u):=[0,u]=\{g \in G: 0 \le g \le u\},$ then $\Gamma(G,u)=(\Gamma(G,u); +,0,u)$ is a pseudo effect algebra, where $+$ is the restriction of the group addition $+$ to $[0,u],$ i.e. $a+b$ is defined in $\Gamma(G,u)$ for $a,b \in \Gamma(G,u)$ iff $a+b \in \Gamma(G,u).$ Then $a^-=u-a$ and $a^\sim=-a+u$ for all $a \in \Gamma(G,u).$ A pseudo effect algebra which is isomorphic to some $\Gamma(G,u)$ for some po-group $G$ with $u>0$ is said to be an {\it interval pseudo effect algebra}.

If $+$ is commutative, i.e. $a+b$ is defined in $E$ iff  $b+a$ is defined in $E$ and $a+b=b+a,$ then $E$ is an {\it effect algebra} in the sense of \cite{FoBe}. For more information on effect algebras, we recommend \cite{DvPu}.

A pseudo effect algebra $E$ is said to be {\it symmetric} if $a^-=a^\sim $ for all $a \in E.$ We note that if $E$ is symmetric, then $E$ is not automatically an effect algebra. Indeed, if $G$ is a po-group that is not Abelian, then for the lexicographic product $\mathbb Z \lex G$ of the po-group $\mathbb Z$ with $G$ we have
$E=\Gamma(\mathbb Z\lex G, (1,0))$ is a symmetric pseudo effect algebra that is not an effect algebra.

A more general structure than pseudo effect algebras is the class of generalized pseudo effect algebras introduced in \cite{DvVe3, DvVe4}. A structure  $(E;+,0),$  where $+$ is a partial binary
operation and 0 is a  constant, is called a {\it generalized pseudo-effect algebra} (or a GPEA for short) if, for all $a,b,c \in E,$ the following hold:
\begin{enumerate}
\item[{\rm (GP1)}] $ a+ b$ and $(a+  b)+ c $ exist if and only if
$ b+ c$ and $a+ ( b+ c) $ exist, and in this case,
$(a+  b)+ c =a+ ( b+ c)$;

\item[{\rm (GP2)}] if $ a+ b$ exists, there are elements $d,e\in E$ such
that $a+ b=d+ a=b+ e$;

\item[{\rm (GP3)}] if $ a+ b$ and $a+ c $ exist and are equal, then
$b=c.$ If $b+ a$ and $c+ a $ exist and are equal, then $b=c$;

\item[{\rm (GP4)}] if $ a+ b$  exists and $a+ b=0$, then $a=b=0$;

\item[{\rm (GP5)}] $ a+ 0$ and $0+ a $ exist and both are equal to $a.$
\end{enumerate}

A GPEA $E$ is {\it trivial} if $E=\{0\}$ and it is {\it non-trivial} if $|E|\ge 2$.

In the same way as for pseudo effect algebras,  we introduce a binary relation $\le$ in a GPEA $E$: For $a,b\in E,$ we define $a\le b$ if
and only if there is an element $c\in E$ such that $a+ c=b.$ Equivalently, there exists an element $d\in E$ such that $d+ a=b.$ Then $\le$ is a partial order on $E$.

If the partial operation $+$ on $E$ is commutative, then a GPEA $E$ is said to be a {\it generalized effect algebra}, GEA for short.

We introduce also two partial binary operations $\minusli$ and $\minusre$ on a GPEA $E$ in the same way as for pseudo effect algebras: For any $a,b\in E$, $a\minusre b$ is defined if and only if
$b\minusli a$ is defined if and only if $a\le b$, and in such a case we have $(b\minusli a)+ a=b= a+ (a\minusre b).$  Then $a=(b\minusli a)\minusre b=b\minusli(a\minusre b).$

For example, if $G$ is a po-group and  $G^+:=\{g \in G: g\ge 0\}$ is the positive cone of $G,$ then $(G^+;0,+)$ is a GPEA where $+$ is the restriction of the group addition $+$ in $G$ to $G^+$ in the obvious sense. Similarly, let $G_0$ be a non-empty subset of $G^+$ such that for all $a, b \in  G_0$, where $b \le a$, also $a - b, -b + a \in  G_0$. Then $(G_0;+,0),$ where $+$ is the group addition restricted to those pairs of elements of $G_0$ whose sum is again in $G_0$, is a GPEA, see \cite[Ex. 2.3]{DvVe4}. In particular, if $u\in G^+$, $u>0,$ and $[0,u):=\{g \in G: 0\le g<u\}$, then $([0,u);+0)$ is a GPEA which has no top element, therefore, it is not a pseudo effect algebra.  If $E$ is a PEA, then the same is true for $(E\setminus\{1\}; +,0)$. (In the latter two cases, $+$ is the restriction of the original addition in the natural sense.)

Let $E,F$ be GPEAs. A mapping $h:E\to F$ is said to be a {\it homomorphism} of PGEAs if  $h(a+b)=h(a)+h(b)$ whenever $a+b$ is defined in $E$; we note that $h(0)=0$. If $E,F$ are PEAs, then a mapping $h:E \to F$ such that (i) $h(a+b)=h(a)+h(b)$ whenever $a+b$ is defined in $E$, and (ii) $h(1)=1$ is said to be a {\it homomorphism} of PEAs. If $h$ and $h^{-1}$ are homomorphisms, then $h$ is said to be an {\it isomorphism}.

A subset $A$ of a GPEA $E$ is a {\it sub-GPEA} of $E$ if (i) $0 \in A$, and (ii) if from three elements $x, y, z \in E$ such that $x + y = z$ at least two are in $A$, then all $x, y, z \in A$. If $E$ is a PEA, then $A\subseteq E$ is a {\it sub-PEA} of $E$ iff (i) $1\in A,$ (ii) $a\in A$ implies $a^-,a^\sim \in A$, (iii) if $a,b \in A$ and $c = a+b,$ then $c\in A$.

We remind that a subset $P_0$ of a poset $P$ is called {\it convex} if, for any two elements  $a,b \in P_0$ and any $c\in P$ such that $a\le c\le b$, we have $c\in P_0$. We note that if $G$ is a po-group and $G_0$ is a convex subset of $G^+$ containing $0$, then $(G_0; +,0)$ is a GPE-algebra.

We recall that a poset $P$ is said to be {\it directed} (more precisely {\it upwards directed)} if, for all $a,b \in P,$ there is an element $c \in P$ such that $a,b \le c$.

By \cite{XLGRD}, a GPEA or a PEA  $E$ is said to be {\it weakly commutative} if $x+y$ is defined in $E$ iff $y+x$ is defined in $E$. For example, if $G$ is a po-group that is not Abelian, then $G^+$ is a weakly commutative GPEA that is not commutative. It is easy to show that a pseudo effect algebra $E$ is weakly commutative if and only if $E$ is symmetric.

We say that a GPEA $E$ satisfies

\begin{enumerate}
\item[(i)]
the {\it Riesz Interpolation Property} (RIP for short) if, for $a_1,a_2, b_1,b_2\in E,$  $a_1,a_2 \le b_1,b_2$  implies there exists an element $c\in E$ such that $a_1,a_2 \le c \le b_1,b_2;$

\item[(ii)]
\RDP$_0$  if, for $a,b,c \in E,$ $a \le b+c$, there exist $b_1,c_1 \in E,$ such that $b_1\le b,$ $c_1 \le c$ and $a = b_1 +c_1;$

\item[(iii)]
\RDP\  if, for all $a_1,a_2,b_1,b_2 \in E$ such that $a_1 + a_2 = b_1+b_2,$ there are four elements $c_{11},c_{12},c_{21},c_{22}\in E$ such that $a_1 = c_{11}+c_{12},$ $a_2= c_{21}+c_{22},$ $b_1= c_{11} + c_{21}$ and $b_2= c_{12}+c_{22};$ this property will be formally denoted by the following table:

$$
\begin{matrix}
a_1  &\vline & c_{11} & c_{12}\\
a_{2} &\vline & c_{21} & c_{22}\\
  \hline     &\vline      &b_{1} & b_{2}
\end{matrix}\ \ ;
$$

\item[(iv)]
\RDP$_1$  if, for all $a_1,a_2,b_1,b_2 \in E$ such that $a_1 + a_2 = b_1+b_2,$ there are four elements $c_{11},c_{12},c_{21},c_{22}\in E$ such that $a_1 = c_{11}+c_{12},$ $a_2= c_{21}+c_{22},$ $b_1= c_{11} + c_{21}$ and $b_2= c_{12}+c_{22}$, and $0\le x\le c_{12}$ and $0\le y \le c_{21}$ imply  $x+y=y+x;$

\item[(v)]
\RDP$_2$  if, for all $a_1,a_2,b_1,b_2 \in E$ such that $a_1 + a_2 = b_1+b_2,$ there are four elements $c_{11},c_{12},c_{21},c_{22}\in E$ such that $a_1 = c_{11}+c_{12},$ $a_2= c_{21}+c_{22},$ $b_1= c_{11} + c_{21}$ and $b_2= c_{12}+c_{22}$, and $c_{12}\wedge c_{21}=0.$

\end{enumerate}

If, for $a,b \in E,$ we have for all $0\le x \le a$ and $0\le y\le b,$ $x+y=y+x,$ we denote this property by $a\, \mbox{\rm \bf com}\, b.$

If we change a GPEA $E$ to a positive cone $G^+$ of a po-group $G,$ we say that $G$ satisfies the analogous type of the Riesz Decomposition Property.

By \cite[Prop 4.2]{DvVe1} for directed po-groups, we have
$$
\RDP_2 \quad \Rightarrow \RDP_1 \quad \Rightarrow \RDP \quad \Rightarrow \RDP_0 \quad \Leftrightarrow \quad  \RIP,
$$
but the converse implications do not hold, in general.  A directed po-group $G$ satisfies \RDP$_2$ iff $G$ is an $\ell$-group, \cite[Prop 4.2(ii)]{DvVe1}.

The fundamental result of theory of pseudo effect algebras, which is a bridge between pseudo effect algebras satisfying RDP$_1$ and the class of unital po-groups satisfying RDP$_1$, is the following representation theorem \cite[Thm 7.2]{DvVe2}:

\begin{theorem}\label{th:2.1}
For every pseudo effect algebra $E$ with \RDP$_1,$ there is a unique $($up to isomorphism of unital po-groups$)$ unital po-group $(G,u)$ with \RDP$_1$\ such that $E \cong \Gamma(G,u).$

In addition, $\Gamma$ defines a categorical equivalence between the category of pseudo effect algebras with \RDP$_1$ and the category of unital po-groups with \RDP$_1.$
\end{theorem}

A similar result as Theorem \ref{th:2.1}, Theorem \ref{th:2.2} below, holds also for directed PEAs, see \cite[Thm 4.8, Prop 5.3, Thm 6.4]{DvVe4}. To show this, we need the following notions. By the couple $(G,G_0)$ we mean a directed po-group $G$  with a fixed directed convex subset $G_0$ of the positive cone such that $0 \in G_0$ and $G_0$ generates $G$ as a po-group. We note that $G_0$ is always a GPEA which is a subalgebra of $G^+$. If $u$ is a strong unit of $G$, then $(G,[0,u])$ is such an example. $(G,G_0)$ and $(H,H_0)$ are {\it isomorphic}, if there is a po-group isomorphism $f:G\to H$ such that $f(G_0)=H_0.$

By $\mathcal{GPEA}$ we understand the category whose objects are directed GPEAs with RDP$_1$ and morphisms are homomorphisms of GPEAs. By $\mathcal{POG}$ we mean the category whose objects are couples $(G,G_0),$ where $G$ is a directed po-group with RDP$_1$ with a fixed directed convex subset $G_0\subseteq G^+$, $0\in G_0,$ $G_0$ generates $G,$ and morphisms from $(G,G_0)$ into $(H,H_0)$ are homomorphisms $h: G\to H$ of po-groups such that $h(G_0)\subseteq H_0.$ Now by $\Gamma(G,G_0)$ we mean the GPEA $G_0:=(G_0;+,0)$ as defined above.

\begin{theorem}\label{th:2.2}
For every directed generalized pseudo effect algebra $E$ with \RDP$_1,$ there is a unique couple $(G,G_0)$ $($up to isomorphism$)$, where $G$ is a directed po-group  with \RDP$_1$\  with a fixed directed convex subset $G_0\subseteq G^+$ generating $G$  such that $E \cong G_0.$

In addition, $\Gamma$ defines a categorical equivalence between the category $\mathcal{GPEA}$ of pseudo effect algebras with \RDP$_1$ and the category $\mathcal{POG}$ of directed po-groups with \RDP$_1.$
\end{theorem}

We note that in the later theorem, $G_0$ is a directed GPEA which is a subalgebra of the GPEA $G^+$.

We finish this section with a note that a PEA $E$ satisfies RDP$_2$ iff $E$ is a lattice and, for all $a,b \in E$, we have $a\minusli (a\wedge b)=(a\vee b)\minusli a$ and $a\minusre (a\vee b)=(a\wedge b)\minusre a$, \cite[Sect 8]{DvVe2}. In addition, there is a unique unital (up to isomorphism of unital $\ell$-groups) $\ell$-group $(G,u)$ such that $E \cong \Gamma(G,u)$, see \cite{Dvu1}.

\section{Kite Pseudo Effect Algebras}%3

In this section we define kite pseudo effect algebras starting with a GPEA. We show when this kind of pseudo effect algebras satisfies different types of the Riesz Decomposition Properties.

Let $E$ be a generalized pseudo effect algebra. We denote by $\overline{E}$ an identical copy of $E$ whose elements are of the form $\bar a$ for all $a\in E$, that is $\overline{E} :=\{\bar a\colon a \in E\}$. We assume that $\bar a= \bar b$ iff $a=b$.

Let $I$ be a set. Define an algebra whose
universe is the set $E^I \uplus (\overline E)^I,$ where $\uplus$ denotes the union of disjoint sets.  Let $\lambda,\rho: I \to I$ be bijections. We define two special elements $0 = 0^I:=\langle 0_j\colon j \in I\rangle$ and $1= \bar 0^I:= \langle \bar 0_i\colon i \in I\rangle$, where $0_j=0=0_i$ for all $j,i \in I$. The elements of $E^I$ will be denoted by $\langle f_j\colon j \in I\rangle$ and ones of $(\overline E)^I$ by $\langle \bar a_i\colon i \in I\rangle$, where $f_j,a_i \in E$ for all $i,j\in I$.

We say that a GPEA $E$ is $\lambda,\rho$-{\it weakly commutative}, where $\lambda,\rho:I \to I$ are bijections, if given sequences $\langle f_j: j\in I\rangle, \langle g_j: j\in I\rangle$ and  sequences $\langle \bar a_i\colon i\in I\rangle, \langle \bar b_i\colon i\in I\rangle$ of elements of $E$ and $\overline {E}$, respectively, %and any element $a \in E,$
we have the equivalences: (i) Given $i \in I,$ $f_{\rho^{-1}(i)}+a_i$ is defined in $E$  iff $a_i + f_{\lambda^{-1}(i)}$ is defined in $E,$ and (ii) given  $i\in I$, $g_{\lambda^{-1}(i)} +b_i$ is defined in $E$ iff $b_i+ g_{\rho^{-1}(i)}$ is defined in $E$.  For example, if $E=G^+$ for some po-group $G$ or $a+b$ is defined in $E$ for all $a,b\in E$ (that is, $+$ is total), then $E$ is $\lambda,\rho$-weakly commutative. The same is true if $\lambda =\rho$ and $E$ is a GEA or a weakly commutative GPEA. In addition, it is possible to show that if $\lambda\ne \rho$, then $E$ is $\lambda,\rho$-weakly commutative iff $a+b$ is defined in $E$ for all $a,b\in E$.

\begin{theorem}\label{th:3.1}
Let $\lambda,\rho:I\to I$ be bijections and $E$ be a $\lambda,\rho$-weakly commutative generalized pseudo effect algebra. Let us endow the set $(E)^I \uplus (\overline E)^I$ with $0=0^I,$ $1=(\bar 0)^I$ and with a partial operation $+$ as follows,

$$\langle \bar a_i\colon i\in I\rangle + \langle \bar b_i\colon i\in I\rangle \eqno(I)$$
is not defined;

$$ \langle \bar a_i\colon i\in I\rangle + \langle f_j\colon j\in I\rangle:= \langle \overline{f_{\rho^{-1}(i)}\minusre a_i}\colon i\in I\rangle \eqno(II)
$$
whenever  $f_{\rho^{-1}(i)}\le a_i$  for all $i \in I;$

$$ \langle f_j\colon j\in I\rangle+ \langle \bar a_i\colon i\in I\rangle  := \langle \overline{a_i \minusli f_{\lambda^{-1}(i)}} \colon i\in I\rangle \eqno(III)
$$
whenever  $f_{\lambda^{-1}(i)} \le a_i$  for all $i \in I,$

$$
\langle f_j\colon j\in I\rangle + \langle g_j\colon j\in I\rangle:= \langle f_j+ g_j\colon j\in I\rangle
\eqno(IV)
$$
whenever $f_j+g_j$ is defined in $E$ for all $j \in I$.

Then the partial algebra $K^{\lambda,\rho}_I(E):=((E)^I \uplus (\overline E)^I; +,0,1)$ is a pseudo effect algebra.

In addition, for the negations in the kite $K_{I}^{\lambda,\rho}(E),$ we have
\begin{align*}
\langle \bar a_i\colon i\in I\rangle^\sim &= \langle a_{\rho(j)}\colon j\in I\rangle\\
\langle \bar a_i\colon i\in I\rangle^-&= \langle a_{\lambda(j)}\colon j\in I\rangle\\
\langle f_j\colon j\in I\rangle^\sim &= \langle \bar f_{\lambda^{-1}(i)}\colon i\in I\rangle\\
\langle f_j\colon j\in I\rangle^-&= \langle \bar f_{\rho^{-1}(i)}\colon i\in I\rangle,
\end{align*}
and $K^{\lambda,\rho}_I(E)$ is symmetric if and only if $\lambda=\rho$.
\end{theorem}

\begin{proof}
(i) To prove  associativity of $+$, we have eight cases, and it is enough to prove only the following case  because all others are simple.

\begin{align*}
(\langle f_j\colon j\in I\rangle + \langle \bar b_i\colon i\in I\rangle)&+
\langle h_j\colon j\in I\rangle
=\langle \overline{b_i \minusli f_{\lambda^{-1}(i)}}\colon i\in I\rangle + \langle h_j\colon j\in I\rangle\\
&=\langle \overline {h_{\rho^{-1}(i)} \minusre (b_i \minusli f_{\lambda^{-1}(i)})}\colon i\in I\rangle.\\
\end{align*}
If the left hand side of the first line exists, then $f_{\lambda^{-1}(i)}\le b_i$ and $h_{\rho^{-1}(i)}\le b_i\minusli f_{\lambda^{-1}(i)}.$ The second inequality entails $h_{\rho^{-1}(i)}\le b_i \minusli f_{\lambda^{-1}(i)}^{-1} \le b_i$ for $i \in I$ so that $ \langle \bar b_i\colon i\in I\rangle+
\langle h_j\colon j\in I\rangle$ and $\langle f_j\colon j\in I\rangle + (\langle \bar b_i\colon i\in I\rangle+
\langle h_j\colon j\in I\rangle)$ are defined, and the left hand side coincides with

$$\langle f_j\colon j\in I\rangle + (\langle \bar b_i\colon i\in I\rangle+
\langle h_j\colon j\in I\rangle).
$$

Conversely, let the later elements be defined, then $h_{\rho^{-1}(i)} \le b_i$ and $f_{\lambda^{-1}(i)} \le h_{\rho^{-1}(i)}\minusre b_i.$ Then $h_{\rho^{-1}(i)} + f_{\lambda^{-1}(i)} \le b_i$ and $f_{\lambda^{-1}(i)} \le h_{\rho^{-1}(i)} + f_{\lambda^{-1}(i)} \le b_i,$ we have $h_{\rho^{-1}(i)}\le b_i \minusli f_{\lambda^{-1}(i)}.$ This yields that the elements $\langle f_j\colon j\in I\rangle + \langle \bar b_i\colon i\in I\rangle$ and $(\langle f_j\colon j\in I\rangle + \langle \bar b_i\colon i\in I\rangle)+
\langle h_j\colon j\in I\rangle$ are defined, and the last expression coincides with $(\langle f_j\colon j\in I\rangle + \langle \bar b_i\colon i\in I\rangle)+\langle h_j\colon j\in I\rangle$.

(ii) If we define the left and right negations  as it is indicated,  we have  $x^- + x=1=x+x^\sim.$ Their uniqueness can be proved thanks to uniqueness properties in the GPEA $E$.

(iii) Assume $x+y$ is defined. For example, let $\langle \bar a_i\colon i\in I\rangle + \langle f_j\colon j\in I\rangle:= \langle \overline{f_{\rho^{-1}(i)}\minusre a_i}\colon i\in I\rangle$ be defined.

Since $a_i = f_{\rho^{-1}(i)}+(f_{\rho^{-1}(i)}\minusre a_i)\in E$, from the $\lambda,\rho$-weak commutativity we conclude that
the element $b_i = (f_{\rho^{-1}(i)}\minusre a_i)+ f_{\lambda^{-1}(i)}$ is defined in $E$, too. It is easy to show that the element
$h_j=(f_{\rho^{-1}(\lambda(j))}\minusre a_{\lambda(j)})\minusre a_{\lambda(j)}$ is defined in $E$. Therefore,
$$
\langle h_j\colon j\in I\rangle + \langle \bar a_i\colon i\in I\rangle= \langle \bar a_i\colon i\in I\rangle + \langle f_j\colon j\in I\rangle =
\langle f_j\colon j\in I\rangle + \langle \bar b_i\colon i\in I\rangle.
$$

In a dual way, we proceed with the case $\langle g_j\colon j \in I\rangle + \langle \bar b_i \colon i \in I\rangle$; the third case is evident.

(iv) Let $1 +x$ or $x+1$ be defined, then it easy to show that $x=0.$

Summarizing (i)--(iv), we see $K^{\lambda,\rho}_I(E)$ is a pseudo effect algebra.

Using the negations, we see that $K^{\lambda,\rho}_I(E)$ is symmetric iff $\lambda=\rho$.
\end{proof}

\begin{remark}\label{re:3.2}
{\rm (1) From Theorem \ref{th:3.1} and from its proof of (iii), we see that if $E = G^+$ for some po-group $G$, then $\langle f_j\colon j \in J\rangle \le \langle \bar a_i\colon i \in I\rangle$. In a general case, $\langle f_j\colon j \in J\rangle \le \langle \bar a_i\colon i \in I\rangle$ if and only if $a_i + f_{\lambda^{-1}(i)}$ is defined in $E$ for all $i \in I$, equivalently, $f_{\rho^{-1}(i)}+a_i$ is defined in $E$ for all $i \in I$.

(2) $\langle f_j\colon j \in J\rangle \le \langle \bar a_i\colon i \in I\rangle$ for all $\langle f_j\colon j \in J\rangle, \langle \bar a_i\colon i \in I\rangle$ iff $+$ is a total operation. Consequently, e.g. if $E=[0,1]$, $|I|=1$, then $\langle 0,6\rangle \not\le \langle \overline{0,6}\rangle$. }
\end{remark}

The resulting pseudo effect algebra $K^{\lambda,\rho}_I(E)$ from Theorem \ref{th:3.1} is said to be a {\it kite pseudo effect algebra} or more precisely a {\it kite pseudo effect algebra} of $E$. We note that Theorem \ref{th:3.1} generalizes the construction of kite pseudo effect algebras studied in \cite{DvuK,DvHo}. More precisely, if $E=G^+,$ where $G$ is a po-group, then  $G^+$ is $\lambda,\rho$-weakly commutative, and $K_{I}^{\lambda,\rho}(G^+)$ and the kite $K_{I}^{\lambda,\rho}(G)_{ea}$ defined in \cite{DvuK, DvHo} coincide.

We present some examples of kites which are connected with the case when $E$ is the positive cone of some po-group. We start with the case $E=\{0\},$ then $E$ is $\lambda,\rho$-weakly commutative and $K^{\lambda,\rho}_I(E)$ is a two-element Boolean algebra.

\begin{example}\label{ex:3.5}
Let $|I|=n$ for $n\ge 2,$ and $\lambda(i)=i,$ $\rho(i) = i-1\ (\mathrm{mod}\, n).$
We put $G_n = \mathbb Z \lex (\mathbb Z^n)$ which is ordered lexicographically. We define the addition $*$ on $G_n$ as follows

$$
(m_1,x_0,\ldots,x_{n-1})*(m_2,y_0,\ldots,y_{n-1})=(m_1+m_2,x_0+y_{0+m_1},\ldots,
x_{n-1}+y_{n-1+m_1}),
$$
where addition of the subscripts is performed by $\mathrm{mod}\, n.$
Then $G_n$ with $*$ is an $\ell$-group with the inverse given by $-(m,a_0,\ldots,a_{n-1})=(-m,-a_{-m},\ldots,-a_{n-1-m}),$ and the element $u_n=(1,0,\ldots,0)$ is a strong unit.
Then $K_{I}^{\rho,\lambda}(\mathbb Z^+)$ is isomorphic to the pseudo effect algebra $\Gamma(G_n,u_n)$ with \RDP$_2.$
\end{example}

\begin{example}\label{ex:3.8}
Let $I=\mathbb Z$ and put $\lambda(i)=i$ and $\rho(i)=i-1,$ $i \in I.$ Then the kite pseudo effect algebra $K_\mathbb Z^{\lambda,\rho}(\mathbb Z^+)$  satisfies \RDP$_2.$

Define $W(\mathbb Z):=\mathbb Z\lex \mathbb Z^\mathbb Z,$ and let  multiplication $*$ on it be defined as follows: $(m_1,x_i)*(m_2,y_i)=(m_1+m_2, x_i+y_{i+m_1}).$ Then $(W(\mathbb Z);(0),*)$ is an $\ell$-group, called the wreath product of $\mathbb Z$ by $\mathbb Z$ {\rm \cite[Ex 35.1]{Dar}}, with strong unit $u=(1,(0))$, and the kite pseudo effect algebra $K_\mathbb Z^{\lambda,\rho}(\mathbb Z^+)$ is isomorphic to $\Gamma(W(\mathbb Z),u).$
\end{example}

The following result describes the Riesz Decomposition Properties of kite pseudo effect algebras depending on original $E$.

\begin{theorem}\label{th:3.5}
Let a set $I,$ bijections $\lambda,\rho: I \to I$ be given and let $E$ be a directed GPEA.
{\rm (1)} If $E$ is $\lambda,\rho$-weakly commutative and the kite pseudo effect algebra $K^{\lambda,\rho}_I(E)$ satisfies \RDP $($or \RDP$_1$ and \RDP$_2$, respectively$)$, then $E$ satisfies {\rm RDP} $($or \RDP$_1$ and \RDP$_2$, respectively$)$.

{\rm (2)} If $+$ is total and $E$ satisfies {\rm RDP} $($or \RDP$_1$ and \RDP$_2$, respectively$)$, then $K^{\lambda,\rho}_I(E)$ satisfies \RDP\, $($or \RDP$_1$ and \RDP$_2$, respectively$)$.
\end{theorem}

\begin{proof}
(1) Let $K^{\lambda,\rho}_I(E)$ satisfy RDP and let, for $a_1,a_2,b_1,b_2 \in E,$ we have $a_1+a_2=b_1+b_2.$ Fix an element $i_0\in I$ and for $i=1,2,$ let us define $A_i=\langle f_j^i \colon j\in I\rangle$ by $f_j^i=a_i$ if $j=i_0$ and $f_j^i=0$ otherwise, $B_i= \langle g_j^i \colon j\in I\rangle$ by $g_j^i=a_i$ if $j=i_0$ and $j_j^i=0$ otherwise. Then $A_1+A_2=B_1+B_2$ so that there are $E_{11}=\langle e^{11}_j \colon j\in I\rangle,$ $E_{12}=\langle e^{12}_j \colon j\in I\rangle,$ $E_{21}=\langle e^{21}_j \colon j\in I\rangle,$ $E_{22}=\langle e^{22}_j \colon j\in I\rangle,$ such that $A_1=E_{11}+E_{12},$ $A_2=E_{21}+E_{22},$ $B_1=E_{11}+E_{21},$ and $B_2=E_{12}+E_{22}.$ Using $(IV)$ of Theorem \ref{th:3.1}, from $j=i_0$ we conclude that the elements $e_{uv}=e^{uv}_{i_0},$ $u,v=1,2,$ form the desired decomposition for $a_1,a_2,b_1,b_2$ which proves $E$ satisfies RDP.

To prove that $E$ satisfies RDP$_1$ if $K^{\lambda,\rho}_I(E)$ satisfies RDP$_1,$  we see that for $j\ne i_0$ we have $e^{12}_j=0=e^{21}_j.$ If now $0\le x\le e_{12}$ and $0\le y\le e_{21}$ Putting $X =\langle x_j \colon j\in I\rangle$ and $Y =\langle y_j \colon j\in I\rangle,$ where $x_j=x,$ $y_j=y$ if $j=i_0$ and $x_j=0$ and $y_j=0$ otherwise, we have $X+Y = Y+X$, consequently, $x+y=y+x$ which proves $e_{12}\, \mbox{\rm \bf com}\, e_{21},$ and $E$ satisfies RDP$_1.$

In the same way we prove the case with RDP$_2.$

(2) Conversely, suppose $E$ satisfies RDP (or RDP$_1$) and $+$ is total. Then $E$ is $\lambda,\rho$-weakly commutative.

(i) If $\langle f_j\colon j\in I\rangle + \langle g_j\colon j\in I\rangle= \langle h_j\colon j\in I\rangle + \langle k_j\colon j\in I\rangle$ from $(IV)$ of Theorem \ref{th:3.1} we conclude that for them we can find an RDP decomposition or an RDP$_1$ one.

(ii) Assume $\langle \bar a_i\colon i\in I\rangle + \langle f_j\colon j\in I\rangle= \langle \bar b_i\colon i\in I\rangle +\langle g_j\colon j\in I\rangle.$ Then $a_i,b_i,f_j,g_j \in E$ for all $i,j \in I.$

Now we will use tables so that we will write the elements of the kite in a simpler way: Instead of $\langle \bar a_i\colon i\in I\rangle$ and $\langle f_j\colon j\in I\rangle$ we use $\langle \bar a_i\rangle$ and $\langle f_j\rangle,$ respectively.

Since $E$ is directed, for any $i\in I$, there is an element $d_i\in E$ such that $a_i,b_i \le d_i$.

From $(II)$ of Theorem \ref{th:3.1}, we have $f_{\rho^{-1}(i)}\minusre a_i= g_{\rho^{-1}(i)}\minusre b_i$ for all $i\in I.$ Then $d_i\minusli (f_{\rho^{-1}(i)}\minusre a_i)$ is defined in $E$. We assert $d_i\minusli (f_{\rho^{-1}(i)}\minusre a_i) = (d_i \minusli a_i) +f_{\rho^{-1}(i)} $ for all $i \in I.$ Indeed, we have $(d_i\minusli a_i) +(a_i\minusli f_{\rho^{-1}(i)})+ f_{\rho^{-1}(i)}= d_i = (d_i\minusli b_i) +(b_i\minusli g_{\lambda^{-1}(i)})+ g_{\lambda^{-1}(i)}.$ Therefore, the element $x = (d_i\minusli a_i)+f_{\rho^{-1}(i)}$ is defined in $E$. Hence
\begin{align*}
d_i\minusli a_i &=x\minusli f_{\rho^{-1}(i)}\\
d_i=(x\minusli f_{\rho^{-1}(i)})+a_i=(x\minusli f_{\rho^{-1}(i)})&+ f_{\rho^{-1}(i)} +(f_{\rho^{-1}(i)}\minusre a_i) =x+(f_{\rho^{-1}(i)}\minusre a_i)\\
x&=d_i \minusli (f_{\rho^{-1}(i)} \minusre a_i).
\end{align*}
In the same way we can show that $d_i\minusli (g_{\lambda^{-1}(i)}\minusre b_i) = (d_i \minusli b_i) +g_{\lambda^{-1}(i)}$ and $ d_i\minusli (f_{\rho^{-1}(i)}\minusre a_i) = d_i\minusli (g_{\lambda^{-1}(i)}\minusre b_i)$ so that
$$
(d_i\minusli a_i)+f_{\rho^{-1}(i)} = (d_i\minusli b_i)+g_{\lambda^{-1}(i)}
$$
for all $i \in I$.

Using RDP for $E$, there are $c_{iuv}\in E,$ $u,v=1,2,$ such that we have the decomposition tables:

$$
\begin{matrix}
d_i\minusli a_i  &\vline & c_{i11} & c_{i12}\\
f_{\rho^{-1}(i)} &\vline & c_{i21} & c_{i22}\\
  \hline     &\vline      &d_i\minusli b_i
   & g_{\rho^{-1}(i)}
\end{matrix}\ \ .
$$
From this table we have
\begin{align*}
d_i\minusli a_i =c_{i11}+c_{i12}, \quad & d_i\minusli b_i= c_{i11}+c_{i21}\\
d_i=c_{i11}+c_{i12} +a_i,\quad & d_i = c_{i11}+c_{i21} +b_i\\
c_{i12}+a_i=&\ c_{i21}+b_i.
\end{align*}
This yields that we have an RDP decomposition of the kite for case (ii) as follows

$$
\begin{matrix}
\langle \bar a_i\rangle  &\vline & \langle \overline{c_{i12}+a_i}\rangle  & \langle c_{\rho(j)12} \rangle\\
\langle f_j \rangle &\vline & \langle c_{\rho(j)21}\rangle & \langle c_{\rho(j)22}\rangle\\
  \hline     &\vline      & \langle \overline{b}_i \rangle & \langle g_j \rangle
\end{matrix}\ \ .
$$

It is evident that if $E$ satisfies RDP$_1$ (RDP$_2$), the later table gives also an RDP$_1$ decomposition (RDP$_2$ decomposition).

(iii) Assume $\langle f_j\colon j\in I\rangle + \langle \bar a_i\colon i\in I\rangle= \langle g_j\colon j\in I\rangle +\langle \bar b_i\colon i\in I\rangle.$ We follows the ideas from the proof of case (ii). For any $i \in I,$ there is $d_i\in E$ such that $a_i,b_i\le d_i$.  By $(III)$ of Theorem \ref{th:3.1}, we have $a_i\minusli f_{\lambda^{-1}(i)}=b_i\minusli g_{\lambda^{-1}(i)}$. In the same way as in (ii), we can show that
$$
(a_i\minusli f_{\lambda^{-1}(i)})\minusre d_i= f_{\lambda^{-1}(i)}+(a_i\minusre d_i)= (b_i\minusli f_{\lambda^{-1}(i)})\minusre d_i= g_{\lambda^{-1}(i)}+(b_i\minusre d_i).
$$
The RDP holding in $E$ entails the decompositions

$$
\begin{matrix}
f_{\lambda^{-1}(i)}  &\vline & d_{i11} & d_{i12}\\
a_i\minusre d_i &\vline & d_{i21} & d_{i22}\\
  \hline     &\vline      &g_{\lambda^{-1}(i)} & b_i\minusre d_i
\end{matrix}\ \ .
$$
Therefore, we have $a_i+d_{i21}=b_i+d_{i12}$.
This implies an RDP decomposition for  case (iii)

$$
\begin{matrix}
\langle f_j\rangle  &\vline & \langle d_{\lambda(j)11}\rangle  & \langle d_{\lambda(j)12}\rangle\\
\langle \bar a_i \rangle &\vline & \langle d_{\lambda(j)21}\rangle & \langle \overline{a_i+d_{i21}} \rangle \\
  \hline     &\vline      & \langle g_j \rangle & \langle \overline {b}_i \rangle
\end{matrix}\ \ .
$$

This decomposition is also an RDP$_1$ decomposition whenever $E$ satisfies RDP$_1$; the same is true for RDP$_2$.

(iv) Assume $\langle \bar a_i\colon i\in I\rangle + \langle f_j\colon j\in I\rangle= \langle g_j\colon j\in I\rangle +\langle \bar b_i\colon i\in I\rangle.$

By $(II)-(III)$ of Theorem \ref{th:3.1}, we have $f_{\rho^{-1}(i)}\minusre a_i = b_i \minusli g_{\lambda^{-1}(i)}$ for all $i \in I.$ Due to Remark \ref{re:3.2}(2), $\langle g_j\rangle \le \langle \bar a_i\rangle$, and by property (iii) of pseudo effect algebras, there is $\langle \bar k_i\rangle \in K^{\lambda,\rho}_I(E)$ such that $\langle \bar a_i\rangle = \langle g_j\rangle +\langle \bar k_i\rangle$ so that $k_i = a_i+g_{\lambda^{-1}(i)}\in E$ for all $i\in I$.  In the similar way, there is $\langle \bar l_i\rangle \in K^{\lambda,\rho}_I(E)$ such that  $\langle \bar l_i\rangle + \langle f_j\rangle  =\langle \bar b_i\rangle$. Hence, $l_i = f_{\rho^{-1}(i)}+b_i \in E$ for all $i \in I$. Since $f_{\rho^{-1}(i)}\minusre a_i = b_i \minusli g_{\lambda^{-1}(i)},$ we have $a_i+g_{\lambda^{-1}(i)}= f_{\rho^{-1}(i)}+b_i$ for all $i \in I$.

Therefore, we can use the decomposition table

$$
\begin{matrix}
\langle \bar a_i\rangle  &\vline & \langle g_j\rangle  & \langle \overline{a_i+g_{\lambda^{-1}(i)}}\rangle\\
\langle f_j \rangle &\vline & \langle 0_j\rangle & \langle f_j\rangle\\
  \hline     &\vline      & \langle g_j \rangle & \langle \overline{b}_i \rangle
\end{matrix}\ \ ,
$$
where $0_j=0$ for all $j \in I$, to prove RDP. The table gives also an RDP$_1$ (RDP$_2$ decomposition) decomposition table whenever $E$ satisfies RDP$_1$ (RDP$_2$).
\end{proof}

We note that if $E$ satisfies RDP$_1$, then by Theorem \ref{th:3.5}, the kite pseudo effect algebra $K^{\lambda,\rho}_I(E)$ satisfies also RDP$_1$, so that by Theorem \ref{th:2.1}, there is a unique (up to isomorphism) unital po-group $(G,u)$ such that $K^{\lambda,\rho}_I(E)\cong \Gamma(G,u)$. Please, find it. This was an open problem in \cite{DvuK} even for the case when $E$ is the positive cone of some po-group with RDP$_1$.

%LOOK AT RDP$_1$ od GPEAs and po-groups with RDP$_1$ similarly for RDP$_2$

\begin{proposition}\label{pr:3.6}
Let $E$ be a GPEA such that $a+b$ exists in $E$ for all $a,b\in E$. The kite pseudo effect algebra $K^{\lambda,\rho}_I(E)$ satisfies \RDP$_0$ if and only if $E$ satisfies \RDP$_0.$
\end{proposition}

\begin{proof}
Using $(IV)$ of Theorem \ref{th:3.1}, it is evident that if $K^{\lambda,\rho}_I(E)$ satisfies RDP$_0,$ so does $E.$

Conversely, let $E$ be with RDP$_0.$  We note $\langle x_j\colon j\in I\rangle \le \langle y_j\colon j\in I\rangle$ iff $x_j\le y_j$ for all $j \in I.$

(i) Let $\langle f_j\colon j\in I\rangle \le \langle g_j\colon j\in I\rangle + \langle h_j\colon j\in I\rangle= \langle g_j+h_j\colon j\in I\rangle$ which implies that for all $j\in I,$ there are $g_{j1},h_{j1} \in E$ such that $g_{j1}\le g_j,$  $h_{j1}\le h_j$ and $f_j= g_{j1}+h_{j1}.$ Then $\langle f_j\colon j\in I\rangle = \langle g_{j1}\colon j\in I\rangle + \langle h_{j1}\colon j\in I\rangle$ and $\langle g_{j1}\colon j\in I\rangle \le \langle g_{j}\colon j\in I\rangle$ and $\langle h_{j1}\colon j\in I\rangle \le \langle h_{j}\colon j\in I\rangle.$

(ii) Let $\langle f_j\colon j\in I\rangle \le \langle \bar a_i\colon i\in I\rangle + \langle g_j\colon j\in I\rangle.$ Then $\langle f_j\colon j\in I\rangle= \langle f_j\colon j\in I\rangle + \langle 0_j\colon j\in I\rangle,$ and $\langle f_j\colon j\in I\rangle \le \langle \bar a_i\colon i\in I\rangle$ and $\langle 0_j\colon j\in I\rangle \le \langle g_j\colon j\in I\rangle.$

In the analogous way we can prove the case $\langle f_j\colon j\in I\rangle \le   \langle g_j\colon j\in I\rangle +\langle \bar a_i\colon i\in I\rangle.$

(iii) Let $\langle \bar a_i\colon i\in I\rangle \le \langle \bar b_i\colon i\in I\rangle + \langle f_j\colon j\in I\rangle.$  We note that $\langle \bar x_i\colon i \in I\rangle \le \langle \bar y_i\colon i \in I\rangle$ iff $x_i\ge y_i$ for all $i \in I.$
For each $i \in I,$ we have $a_i \ge f_{\rho^{-1}(i)}\minusre b_i$ and $b_i \le f_{\rho^{-1}(i)}+a_i$.
Using RDP$_0$ for $E,$ we can find positive elements $a_{i1}\le a_i$ and $f_{\rho^{-1}(i)1}\le f_{\rho^{-1}(i)}$ such that $b_i=f_{\rho^{-1}(i)1}+ a_{i1}.$ Then we have $\langle \bar a_i\colon i\in I\rangle = \langle \overline{f_{\rho^{-1}(i)1}+a_i} \colon i\in I\rangle
+ \langle f_{j1} \colon j\in I\rangle,$ and $\langle \overline{ f_{\rho^{-1}(i)1}+a_i} \colon i\in I\rangle \le \langle \bar b_i\colon i\in I\rangle$ and $\langle f_{j1} \colon j\in I\rangle \le \langle f_{j} \colon j\in I\rangle.$

In a dual way, we proceed also with the case $\langle \bar a_i\colon i\in I\rangle \le   \langle f_j\colon j\in I\rangle + \langle \bar b_i\colon i\in I\rangle.$

Summing up all cases, we have that $K^{\lambda,\rho}_I(E)$ satisfies RDP$_0.$
\end{proof}

Let $x$ be an element of the pseudo effect algebra $E$ and $n\ge 0$ be an integer. We define $0x:=0,$ $1x:=x,$ and $(n+1)x:= nx +x$ whenever $nx$ and $nx +x$ are defined in $E.$ An element $x\in E$ is said to be {\it infinitesimal} if $nx$ exists in $E$ for any integer $n \ge 1.$ We denote by $\Infinit(E)$ the set of infinitesimal elements of $E.$

Perfect pseudo effect algebras are characterized as those PEAs where every element is either infinitesimal or co-infinitesimal, see e.g. \cite{DXY, DvKr}.  Such PEAs are often of the form $\Gamma(\mathbb Z\lex G,(1,0))$ for some directed po-group $G,$ c.f. \cite[Thm 5.2]{DvuK}.

If $A,B$ are two subsets of a pseudo effect algebra $E,$ $A+B:=\{a+b: a \in A, b \in B, a+b \in E\},$ and we say that $A+B$ is {\it defined} in $E$ if $a+b$ exists in $E$ for each $a \in A$ and each $b \in B.$ We write $A\leqslant B$ whenever $a\le b$ for all $a\in A$ and all $b \in B.$ In addition, we write $A^-:=\{a^-: a \in A\}$ and $A^\sim :=\{a^\sim : a \in A\}.$

We note that an {\it ideal} of a generalized pseudo effect algebra $E$ is any non-empty subset $I$ of $E$ such that (i) if $x,y \in I$ and $x+y$ is defined in $E,$ then $x+y \in I,$ and (ii) $x\le y \in I$ implies $x\in I.$ An ideal $I$ is {\it maximal} if it is a proper subset of $E$ and it is not a proper subset of any ideal $J\ne E.$ An ideal $I$ is {\it normal} if $x+I=I+x$ for any $x \in E,$ where $x+I:=\{x+y: y \in I,\ x+y $ exists in $E\}$ and in the dual way  we define $I+x.$ The set $\{0\}$ is always a normal ideal, and an ideal $I$ is {\it non-trivial} if $I \ne \{0\}.$

We say that a mapping $s:E \to[0,1]$,  is a {\it state}, where $E$ is a PEA, if (i) $s(a+b)=s(a)+s(b)$ whenever $a+b$ is defined in $E,$ and (ii) $s(1)=1.$ A state $s$ is {\it extremal} if from $s=\alpha s_1+(1-\alpha) s_2,$ where $s_1,s_2$ are states on $E$ and $\alpha \in (0,1),$ we conclude $s=s_1=s_2.$ A state models a finitely additive probability measure. We denote by $\mathcal S(E)$ the set of all states on $E.$ It can happen that $\mathcal S(E)$ is empty. In general, $\mathcal S(E)$ is either empty, or a singleton or an infinite set. We recall that if $E$ is an effect algebra with RDP, then $E$ has at least one state.

If $s$ is a state on $E,$ then the {\it kernel} of $s,$ i.e. the set $\Ker(s):=\{a\in E: s(a)=0\},$ is a normal ideal of $E.$

We say that a pseudo effect algebra $E$ is {\it perfect} if there are two subsets $E_0,E_1$ of $E$ such that $E_0\cap E_1=\emptyset$ and $E=E_0\cup E_1$ such that
\begin{enumerate}
\item[(a)] $E_i^- =E_i^\sim = E_{1-i},$ $i=0,1,$
\item[(b)] if $x \in E_i,$ $y\in E_j$ and $x+y$ is defined in $E,$ then $i+j\le 1$ and $x+y \in E_{i+j}$ for $i,j=0,1,$
\item[(c)]    $E_0+ E_0$ is defined in $E.$
\end{enumerate}
In such a case, we write $E=(E_0,E_1).$ By \cite[Prop 5.1]{DvuK}, $E_0$ consists of all infinitesimal elements of $E$, i.e. $E_0=\Infinit(E)$,  $E$ has a unique state, say $s$, and $s(E_0)=0$ and $s(E_1)=1.$

The following result was established in \cite[Thm 5.3]{DvuK} for the case when $E=G^+.$

\begin{theorem}\label{th:3.7}
Assume that a set $I,$ a bijection $\lambda: I \to I$ and  a directed weakly commutative GPEA $E$ with \RDP$_1$ are given. Assume that $+$ on $E$ is a total operation. There is a unique (up to isomorphism) directed po-group $G^\lambda_I$ with \RDP$_1$ such that
the perfect kite pseudo effect algebra $K^{\lambda,\lambda}_I(E)$ is isomorphic to $\Gamma(\mathbb Z\lex G^\lambda_I,(1,0)).$ In particular, $K^{\lambda,\lambda}_I(E)$ is a perfect PEA.
\end{theorem}

\begin{proof}
It is evident that the kite pseudo effect algebra $K^{\lambda,\lambda}_I(E)$ is $\lambda,\lambda$-weakly commutative, and in addition, by negations presented in Theorem \ref{th:3.1}, it is symmetric. It is evident that any kite pseudo effect algebra is perfect.
Since $E$ is with RDP$_1,$ by Theorem \ref{th:3.5}, the kite $K^{\lambda,\lambda}_I(E)$ satisfies RDP$_1.$ Consequently, all conditions of \cite[Thm 5.2]{DvuK} are satisfied, thus there is a unique (up to isomorphism) directed po-group $G^\lambda_I$ with RDP$_1$ such that
$K^{\lambda,\lambda}_I(E)\cong \Gamma(\mathbb Z\lex G^\lambda_I,(1,0)).$
\end{proof}

\begin{remark}\label{re:3.8}
{\rm (1) If the operation $+$ is not total in a GPEA $E$, then Theorem \ref{th:3.7} is not valid. For example if $E=[0,1)$ is the interval or real numbers, then $K^{\lambda,\lambda}_I(E)$ is not a perfect kite pseudo effect algebra.

(2) If the operation $+$ is total in a GPEA $E$, then by \cite[Prop X.1]{Fuc}, $E$ is the positive cone of a po-group.}
\end{remark}

\section{Subdirect Irreducibility of Kite Pseudo Effect Algebras}%4

In this section we generalize the results from \cite{DvuK,DvHo} concerning the subdirect irreducibility of the kite pseudo effect algebras.  We note that in \cite{DvuK,DvHo} these results were established for the special case of a GPEA $E=G^+,$ which is always $\lambda,\rho$-weakly commutative, the present proofs follow the basic steps and ideas of the original proofs from \cite{DvuK,DvHo}. To be self-contained, our proofs are given with fullness with necessary changes.

We remind that by an {\it o-ideal} of a directed po-group $G$ we understand any normal directed convex subgroup $H$ of $G$. If $G$ is a po-group, so is $G/H,$ where $x/H \le y/H$ iff $x\le h_1+y$ for some $h_1\in H$ iff $x \le y+h_2$ for some $h_2 \in H.$
If $G$ satisfies one of RDP's, then $G/H$ satisfies the same RDP, \cite[Prop 6.1]{174}.

We say that an equivalence $\sim$ on  a GPEA $E$ is a {\it congruence}, if for $a_1,a_2,b_1,b_2$ such that $a_1\sim a_2$ and $b_1 \sim b_2$ the existence $a_1+b_1$ and $a_2+b_2$ in $E$ implies $a_1+b_1 \sim a_2+b_2.$  If $I$ is a normal ideal of $E,$ then the relation $\sim_I$ defined on $E$ by
$a\sim_I b$ iff there are $e,f \in I$ such that $a\minusli e=b\minusli f.$  Then $\sim_I$ is a congruence on $E$. In the same way as \cite[Prop 4.1]{174}, we can show that if the GPEA $E$ satisfies RDP$_1,$ then there is a one-to-one correspondence between congruences on $E$ and normal ideals of $E.$ We note that if $\sim$ is a congruence, then $I=\{a \in E\colon a\sim 0\}$ is a normal ideal of $E$.

If a directed  GPEA $E$ satisfies RDP$_1$, by Theorem \ref{th:2.1}, there is a unique (up to isomorphism) couple $(G,G_0),$ where $G$ is a directed po-group with RDP$_1,$ $G_0$ is a GPEA which is a subalgebra of $(G^+;+,0)$ and $G_0$ generates the po-group $G$, such that $E \cong \Gamma(G,G_0).$ In the same way as \cite[Thm 4.2]{185}, we can prove the following statement:

\begin{proposition}\label{pr:4.1}
Let $E=\Gamma(G,G_0)$ for some directed po-group $G$ with \RDP$_1$. If $I$ is an ideal of $E$,
then the set
$$\phi(I)=\{x \in G: \exists \ x_i,y_j \in I,\ x= x_1+\cdots +x_n-y_1-\cdots -y_m\}$$
is a convex subgroup of $G$ generated by $I.$ The set $\phi(I)$ is an o-ideal if and only if $I$ is a normal ideal, and $E/I =\Gamma(G/\phi(I),u/\phi(I)).$ The mapping $I\mapsto \phi(I)$ is a one-to-one correspondence  between normal ideals of $E$ and o-ideals of $G$ preserving set-theoretical inclusion. If $K$ is an o-ideal, then the restriction $K\cap G_0$ is a normal ideal of $E.$
\end{proposition}

\begin{remark}\label{re:2}
{\rm We note that in the same way as \cite[Lem 3.4]{DvHo}, we san prove that of a GPEA $E$ satisfies RDP$_1$ and $I$ is a normal ideal of $E,$ then $E/I$ is a GPEA with RDP$_1$.}
\end{remark}

We say that a GPEA $E$ (or a po-group) is a {\it subdirect product} of a family $(E_t\colon t \in T)$ of GPEAs (po-groups), and we write $E \leq  \prod_{t \in T}E_t$ if there is an injective homomorphism $h\colon E \to \prod_{t \in T}E_t$ such that $\pi_t\circ h(E)=E_t$ for all $t \in T,$ where $\pi_t$ is the $t$-th projection from $\prod_{t \in T}E_t$ onto $E_t.$ In addition,  $E$ is
{\it subdirectly irreducible} if whenever $E$ is a subdirect product of $(E_t: t \in T),$ there exists $t_0 \in T$ such that $\pi_{t_0}\circ h$ is an isomorphism of generalized pseudo effect algebras. It is possible to show that a GPEA $E$ is subdirectly irreducible iff $E$ is either trivial or it possesses the least non-trivial normal  ideal $I$, or equivalently, the intersection of all non-trivial normal ideals of $E$ is a non-trivial normal ideal.

The following result generalizes an analogous one from \cite[Lem 3.2]{DvuK}, where it was proved for PEAs with RDP$_1$; the present proof follows the original proof from \cite[Lem 3.2]{DvuK}.

\begin{lemma}\label{le:7.1}
Every directed generalized pseudo effect algebra with \RDP$_1$ is a subdirect product of subdirectly irreducible directed generalized pseudo effect algebras with \RDP$_1.$
\end{lemma}

\begin{proof}
If $E$ is a trivial GPEA, i.e. $E=\{0\}$, $E$ is subdirectly irreducible. Thus, suppose $E$ is not trivial. By Theorem \ref{th:2.2}, RDP$_1$ entails that $E \cong \Gamma(G,G_0)$ for some po-group $G$ satisfying RDP$_1$, and $G_0$ is a GPEA which is a subalgebra of $G^+$ generating $G$; for simplicity, we assume $E \subseteq \Gamma(G,G_0).$ Given a non-zero element $g\in G,$ let $N_g$ be an o-ideal of $G$ which is maximal with respect to not containing $g$; by Zorn's Lemma, it exists. Then $\bigcap_{g\ne 0} N_g=\{0\}$ and $N_g\cap G_0$ is a normal ideal of $E.$  Therefore, $G \leq \prod_{g\ne 0}G/N_g$ and $E \le \prod_{g\ne 0}E/(N_g\cap G_0).$ In addition, for every $g\ne 0,$ the o-ideal of $G/N_g$ generated by $g/N_g$ is the least non-trivial o-ideal of $G/N_g$ which proves that every $G/N_g$ is subdirectly irreducible. Therefore, every $E/(N_g\cap G_0)$ is a subdirectly irreducible pseudo effect algebra, and by \cite[Prop 4.1]{185}, the quotient pseudo effect algebra $E/(N_g\cap G_0)$ also satisfies RDP$_1.$
\end{proof}

Let $\alpha$ be a cardinal. An element $\langle f_j\colon j \in I\rangle$ is said to be $\alpha$-{\it dimensional},  if $|\{j \in I: f_j \ne e\}|=\alpha.$ In the same  way we define an $\alpha$-dimensional element $\langle \bar a_i\colon i \in I\rangle.$

\begin{proposition}\label{pr:6.3}
Let $I$ be a set and $\lambda,\rho:I \to I$ be bijections.
If $H$ is a normal ideal of a directed $\lambda,\rho$-weakly commutative GPEA $E$,  then $H^I:=\{\langle f_j\colon j\in I\rangle: f_j \in H,\ j\in I\}$ is a normal ideal of the kite pseudo effect algebra
$K^{\lambda,\rho}_I(E).$ In addition, if $H_f^I$ denotes the set of all finite dimensional elements from $H^I,$ then $H_f^I$ is a non-trivial normal ideal of the kite pseudo effect algebra.  In particular, if $H=E,$ then $E^I$ is a maximal ideal of $K^{\lambda,\rho}_I(E)$.

Conversely, if $J$ is a normal ideal of $K^{\lambda,\rho}_I(E),$ $J \subseteq H^I$, then $\pi_j(J)$ is a    normal ideal of $E$ which is a subset of  $H,$ where
$\pi_j$ is the $j$-th projection of $\langle f_j \colon j \in I\rangle \mapsto f_j.$
\end{proposition}

\begin{proof}
It is identical with the proof of \cite[Prop 6.3]{DvuK}, where it was proved for the special case $E=G^+$ for some directed po-group $G.$ We prove only that $E^I$ is a maximal ideal of $K^{\lambda,\rho}_I(E).$ Take $x=\langle \bar f_j\colon j \in I\rangle$ and let $F$ be an ideal of the kite generated by $A^I$ and $x$. By Theorem \ref{th:3.1}, $x^-\in A^I$, so that $x^-+x =1$ therefore, $1\in F$ which proves that $A^I$ is maximal.
\end{proof}

\begin{proposition}\label{pr:6.5}
Let a kite pseudo effect algebra of a non-trivial directed $\lambda,\rho$-weakly commutative GPEA $E,$ $K^{\lambda,\rho}_I(E),$ have the least non-trivial normal ideal. Then $E$ has the least non-trivial normal ideal.
\end{proposition}

\begin{proof}
Suppose the converse, i.e. $E$ has no least non-trivial normal ideal. There exists a set $\{H_t: t \in T\}$ of non-trivial normal ideals of $E$ such that $\bigcap_{t \in T} H_t =\{0\}.$
By Proposition \ref{pr:6.3}, every $H_t^I$ is a normal ideal of the kite. Hence, $\bigcap_{t \in T} H_t \ne \{0\}$ and there is a non-zero element $f=\langle f_j\colon j \in I\rangle\in \bigcap_{t\in T} H_t^I.$ Then, for every index $j \in I,$ $f_j \in H_t$ for every $t \in T$ which yields $f_j=0$ for each $j \in I$ and $f=\langle f_j\colon j \in I\rangle = 0,$ which is a contradiction. Therefore, $E$ has the least non-trivial normal ideal.
\end{proof}

According to \cite{DvKo, DvuK, DvHo}, we say that elements
$i,j\in I$ are {\it connected} if there is an integer $m \ge 0$ such that $(\rho\circ\lambda^{-1})^m(i)= j$ or $(\lambda\circ \rho^{-1})^m(i)= j$; otherwise, $i$ and $j$ are said to be {\it disconnected}.

The relation $i$ and $j$ are connected is an equivalence on $I$. We call this equivalence class a {\it connected component} of $I.$ We denote by $\mathcal C(I)$ the set of all connected components of $I.$

It is noteworthy to recall that if $C$ is a connected component of $I,$ then $\lambda^{-1}(C)= \rho^{-1}(C).$  Indeed, let $i \in C$ and $k=\lambda^{-1}(i).$ Then $j=\rho (k)= \rho \circ \lambda^{-1}(i) \in C.$ Hence, $k=\rho^{-1}(j)$ which proves $\lambda^{-1}(C)\subseteq \rho^{-1}(C).$ In the same way we prove the opposite inclusion. In particular, we have $\lambda(\rho^{-1}(C))= \rho(\lambda^{-1}(C)) = C=\lambda(\lambda^{-1}(C))= \rho(\rho^{-1}(C)).$

\begin{theorem}\label{th:6.6}
Let $I$ be a set and $\lambda,\rho:I \to I$ be bijections and let $E$ be a non-trivial directed $\lambda,\rho$-weakly commutative GPEA. For the kite pseudo effect algebra  $K^{\lambda,\rho}_I(E)$ with \RDP$_1$, the following statements are equivalent:

\begin{enumerate}
\item[{\rm (1)}] $E$ has the least non-trivial normal ideal and for all $i,j\in I$ there exists an integer $m\ge 0$ such that $(\rho\circ\lambda^{-1})^m(i)=j$ or $(\lambda\circ \rho^{-1})^m(i)=j.$
\item[{\rm (2)}] $K^{\lambda,\rho}_I(E)$ has the least non-trivial normal ideal.
\end{enumerate}
\end{theorem}

\begin{proof}

(1) $\Rightarrow$ (2). By Theorem \ref{th:3.5}, the pseudo kite effect algebra $K^{\lambda,\rho}_I(E)$ satisfies RDP$_1.$

Let $H$ be the least non-trivial normal ideal of $E$.
By Proposition \ref{pr:6.3}, $H^I$ and $H_f^I$ are normal ideals of the kite $K^{\lambda,\rho}_I(E).$ Therefore, it is only necessary to show that $H^I_f$ is the least non-trivial normal ideal of the kite. Equivalently, we have to prove that the normal ideal of the kite generated by any nonzero element $\langle f_j\colon j \in I\rangle\in H^I_f$ equals to $H^I_f.$ This is equivalent to show the same for any one-dimensional element from $H^I_f.$ Indeed,
let $f=\langle f_j\colon j \in I\rangle$ be any element from $H^I\setminus \{0\}.$ There is a one-dimensional element $g=\langle g_j\colon j \in I\rangle\in H^I$ such that $0<\langle g_j\colon j \in I\rangle \le \langle f_j\colon j \in I\rangle.$ Hence, $H^I_f=H_0(g)\subseteq H_0(f)\subseteq H^I_f.$

Thus let $g$ be any one-dimensional element from $H^I_f$ and
let $N_0(g)$ be the normal ideal of the kite generated by the element $g.$
Without loss of generality assume $g=\langle g_0,e,\ldots\rangle,$ where $g_0>e,$ $g_0 \in E$; this is always possible by a suitable re-indexing of $I,$ regardless of its cardinality. Then $g_0$ generates $H$, that is $H=\{x\in E: x= -x_1+g_1+x_1+\cdots -x_n+g_n+x_n=y_1+g'_1-y_1+\cdots y_m+g'_m- y_m,\ 0\le g_i,g'_j\le g_0,\  x_i,y_j\in E,\ i=1,\ldots,n,\ j=1,\ldots,m,\  n,m\ge 1\},$ since $H$ is the least non-trivial  normal ideal of $G.$ (We note that we can assume that by Theorem \ref{th:2.2}, $E=\Gamma(G,G_0)$ for some po-group $G$ with RDP$_1,$ so that the elements of the form $-x_1+g_1+x_1$ and $y_1+g'_1-y_1$ exist in $G_0$  as well as in $E$.)   Doing double negations $m$ times of $\langle f_j\colon j \in I\rangle$, we obtain that  either $ \langle f_j\colon j \in I\rangle^{--m}=\langle f_{(\rho\circ \lambda^{-1})^m(j)}\colon j \in I\rangle$ and it belongs to $N_0(g)$ or $\langle f_j\colon j \in I\rangle^{\sim\sim m}=\langle f_{(\lambda\circ\rho^{-1})^m(j)}\colon j \in I\rangle$ which also belongs to $N_0(g).$  Consequently, for any $j \in I,$ there is an integer $m$ such that $(\rho\circ\lambda^{-1})^m(j)=0$ or $(\lambda\circ \rho^{-1})^m(j)=0,$ so that the one-dimensional element whose $j$-th coordinate is $g_0$ is defined in $N_0(g)$ for any $j \in J$; it is either $g^{--m}$ or $g^{\sim\sim m}.$ It is easy to show that $-f+g_0+f$ and $k+g_0-k$ belong to $N_0(g_0)$ for all $g,k\in E,$ which implies, for every $g \in H,$ the one-dimensional element $\langle g,e,\ldots \rangle$ belongs to $N_0(g_0),$ and finally, every one-dimensional element $\langle \ldots, g,\ldots \rangle$  from $H^I_f$ belongs also to $N_0(g_0).$

Now let $J_0=\{j_1,\ldots,j_n\}$ be an arbitrary finite subset of $J,$ $|J_0|=m\ge 1,$ and choose arbitrary $m$ elements $h_{j_1},\ldots,h_{j_m} \in H.$ Define an $m$-dimensional element $g_{J_0}=\langle g_j\colon j \in I\rangle,$ where $g_j=h_{j_k}$ if $j=j_k$ for some $k=1,\ldots,m,$ and $g_j=0$ otherwise. In addition, for $k=1,\ldots,m,$ let $g_k=\langle f_j\colon j \in I\rangle,$ where $f_j=h_{j_k}$ if $j=j_k$ and $f_j=e$ if $j\ne j_k.$ Then $g_{J_0}=g_1+\cdots+g_k\in N_0(g_0).$

Consequently, $N_0(g_0)=H_f^I.$

(2) $\Rightarrow$ (1). By Proposition \ref{pr:6.5}, $E$ has the least non-trivial normal ideal, say $H_0$. Suppose that (1) does not hold. Then for all
$i,j\in I$ and every integer $m\ge 0,$ we have $(\rho\circ\lambda^{-1})^m(i)\ne j$ and $(\lambda\circ \rho^{-1})^m(i)\ne j.$ We can observe that such indices $i$ and $j$ are disconnected.  By the assumption, there are two elements $i_0,j_0\in I$ which are disconnected.  Let $I_0$ and $I_1$ be  maximal sets of mutually connected elements containing $i_0$ and $j_0,$ respectively. Then no element of $I_0$ is connected to any element of $I_1.$

We define $H_0^{I_0}$ as the set of all elements $\langle f_j\colon j \in I\rangle$ such that $j\notin I_0$ implies $f_j=e.$ In a similar way we define $H^{I_1}.$ Then both sets are non-trivial normal ideals of the kite. For example, let $\langle f_j \colon j \in I\rangle\in N_0^{I_0}.$ Then $\langle f_j \colon j \in I\rangle + \langle \bar a_i \colon i \in I\rangle=
\langle \overline{a_i \minusli f_{\lambda^{-1}(i)}} \colon i \in I\rangle.$ If we set $h_j =a_{\rho(j)}\minusli (a\minusli f_{\lambda^{-1}(\rho(j))})= a_{\rho(j)}+ f_{\lambda^{-1}(\rho(j))}- a_{\rho(j)}^{-1}\in E,$
then $\langle h_j \colon j \in I\rangle \in N_0^{I_0}$ and $\langle \bar a_i \colon i \in I\rangle + \langle h_j \colon j \in I\rangle = \langle f_j \colon j \in I\rangle + \langle \bar a_i \colon i \in I\rangle.$ Similarly for the other possibilities. In the same way we can show that $H_0^{I_1}$ is a non-trivial normal ideal of the kite.

On the other hand, we have $H_0^{I_0} \cap H_0^{I_1} = \{0\}$ which contradicts that the kite has the least non-trivial normal ideal.
\end{proof}

%We note that Theorem \ref{th:6.6} is not valid if $E=\{1\}$ because then the corresponding kite is a two-element Boolean algebra 

Theorem \ref{th:6.6} has important two consequences when $I$ is finite or countable describing subdirectly irreducible kite pseudo effect algebras.

\begin{theorem}\label{th:6.7}
Let $|I|=n$ for some $n\ge 0,$ $\lambda,\rho: I \to I$ be bijections and $E$ be a non-trivial directed $\lambda,\rho$-weakly commutative GPEA. If the kite pseudo effect algebra $K^{\lambda,\rho}_I(E)$ with \RDP$_1$ has the least non-trivial normal ideal, then $E$ has the least non-trivial normal ideal and $K^{\lambda,\rho}_I(E)$ is isomorphic to one of:
\begin{enumerate}
\item[{\rm (1)}] $K^{\emptyset,\emptyset}_0(E),$ if $n=0,$ $K^{id,id}_{1}(E)$ if $n=1.$

\item[{\rm (2)}] $K^{\lambda,\rho}_n(E)$ for $n\ge 1$ and $\lambda (i)=i$ and $\rho(i) = i-1\ (\mathrm{mod}\, n).$

\end{enumerate}
\end{theorem}

\begin{proof}
In any case, $E$ has the least non-trivial  normal ideal.

If $I$ is empty, the only bijection from $I$ to $I$ is the empty function. The kite pseudo effect algebra $K^{\emptyset,\emptyset}_0(E)$ is a two-element Boolean algebra.
For $n\ge 1,$ we can assume without loss of generality that $\lambda$ is the identity map and $\rho$ is a permutation of $I=\{0,1,\ldots,n-1\}.$ If $\rho$ is not cyclic, then there are $i,j \in I$ such that $j$ does not belongs to the orbit of $i,$ which means that $i$ and $j$ are not connected which contradicts Theorem \ref{th:6.6}. So $\rho$ must be cyclic. We can then renumber $I=\{0,1,\ldots,n-1\}$ following the $\rho$-cycle, so that $\rho(i) = i-1\ (\mathrm{mod}\, n),$ $i =0,1,\ldots,n-1.$
\end{proof}

The proofs of the following two results are  identical as those from \cite[Lem 6.8, Thm 6.9]{DvuK} for the case $E=G^+$, but to be selfcontained, we present them here in fullness.

Now we show that a necessary condition to be a kite $K^{\lambda,\rho}_I(E)$ with $I$  infinite  subdirectly irreducible  is  $I$ has to be at most countable.

\begin{lemma}\label{le:6.8}
Let $I$ be a set and $\lambda,\rho:I \to I$ be bijections and let $E$ be a non-trivial directed $\lambda,\rho$-weakly commutative GPEA. If the kite pseudo effect algebra $K^{\lambda,\rho}_I(E)$ with \RDP$_1$ has the least non-trivial normal ideal, then $I$ is at most countable.
\end{lemma}

\begin{proof}
Suppose $I$ is uncountable and choose an element $i\in I.$ Consider the set $P(i)=\{(\rho\circ\lambda^{-1})^m(i):\ m \ge 0\} \cup \{(\lambda\circ \rho^{-1})^m(i): m\ge 0\}.$ The set $P(i)$ is at most countable, so there is $j \in I\setminus P(i).$ But $P(i)$ exhausts all finite paths alternating $\lambda$ and $\rho$ starting from $i.$ Then $i$ and $j$ are disconnected which contradicts Theorem \ref{th:6.6}. Hence, $I$ is at most countable.
\end{proof}

Finally, we present kite pseudo effect algebras having the least non-trivial normal ideal when $I$ is infinite. In such a case by Lemma \ref{le:6.8}, $I$ has to be countable.

\begin{theorem}\label{th:6.9}
Let $|I|=\aleph_0,$ $\lambda,\rho: I \to I$ be bijections and let $E$ be a non-trivial directed $\lambda,\rho$-directed GPEA. If the kite pseudo effect algebra $K^{\lambda,\rho}_I(E)$ with \RDP$_1$  has the least non-trivial normal ideal, then $K^{\lambda,\rho}_I(E)$ is isomorphic to $K^{id,\rho}_\mathbb Z(E),$ where $\rho(i)=i-1,$ $i \in \mathbb Z.$
\end{theorem}

\begin{proof}
After re-indexing we can assume that $\lambda$ is the identity on $I.$ If $\rho$ is not cyclic, there would be two elements which should be disconnected which is impossible. Therefore, there is an element $i_0\in I$ such that the orbit $P(i_0):=\{\rho^m(i_0): m \in \mathbb Z\}=I.$ Hence, we can assume that $I=\mathbb Z,$ and $\rho(i)=i-1,$ $i \in \mathbb Z.$ Indeed, if we set $j_m=\rho^{-m}(i_0),$ $m\in \mathbb Z,$ then $\rho(j_m)=j_{m-1}$ and we have $\rho(i)=i-1,$ $i \in \mathbb Z.$
\end{proof}

It is possible to define kite pseudo effect algebras also in the following way. Let $J$ and $I$ be two sets, $\lambda,\rho\colon J \to I$ be two bijections, and $E$ be a $\lambda,\rho$-weakly commutative GPEA. We define $K^{\lambda,\rho}_{J,I}(E)= (E)^J \uplus (\overline{E})^I,$ where $\uplus$ denotes a union of disjoint sets. The elements $0=\langle e_j\colon j \in J\rangle$ and $1=\langle \bar e\colon i \in I\rangle$ are assumed to be the least and greatest elements of $ K^{\lambda,\rho}_{J,I}(E).$  If we define a partial operation $+$ by Theorem \ref{th:3.1}, changing in formulas $(II)-(IV)$ the notation $j \in I$ to $j \in J,$ we obtain that $K_{J,I}^{\lambda,\rho}(E) =(K^{\lambda,\rho}_{J,I}(E);+,0,1) $ with this $+,$ $0$ and $1$ is a pseudo effect algebra, called also a {\it kite pseudo effect algebra.} In particular, $ K^{\lambda,\rho}_{I}(E) = K^{\lambda,\rho}_{I,I}(E).$

Since $J$ and $I$ are of the same cardinality, there is practically no substantial difference between kite pseudo effect algebras of the form $K^{\lambda,\rho}_{I}(E)$ and $K^{\lambda,\rho}_{I,J}(E)$ and all known results holding for the first kind are also valid for the second one. We note that in \cite{DvKo}, the ``kite" structure used two index sets, $J$ and $I$. The second form, which was used also in the proof of \cite[Thm 3.6]{DvHo},  will be practical also for the following result; its proof here follows the main ideas of the original proof.

\begin{lemma}\label{le:7.4}
Let $K_{I}^{\lambda,\rho}(E)$ be the kite pseudo effect algebra with \RDP$_1$ of a directed $\lambda,\rho$-weakly commutative GPEA $E$. Then
$K_{I}^{\lambda,\rho}(E)$ is a subdirect
product of the system of kite pseudo effect algebras with \RDP$_1$ $(K_{I',J'}^{\lambda',\rho'}(E)\colon I'\in \mathcal C(I))$, where $I'$ is any
connected component of $I$, $J'=\lambda^{-1}(I')=\rho^{-1}(I'),$ and $\lambda',\rho'\colon J'\to I'$ are the restrictions of $\lambda$ and $\rho$ to $J'\subseteq I.$
\end{lemma}

\begin{proof}
If $E$ is trivial, the statement is evident. Now, let $E$ be non-trivial.
By the comments before this lemma, we see that $\lambda',\rho'\colon J'\to I'$ are bijections.
Let $I'$ be a connected component of $I$.
Let $N_{I'}$ be the set of all elements
$f = \langle f_j: j \in I\rangle\in (E)^I$
such that  $f_j = 0$ whenever $j\in J'$.
It is straightforward to see that $N_{I'}$ is a normal ideal
of $K_{I}^{\lambda,\rho}(E)$.

We note that under our conditions, $E$ is also $\lambda',\rho'$-weakly commutative, so that $K_{J',I'}^{\lambda',\rho'}(E)$ is by Theorem \ref{th:3.1} again a pseudo effect algebra.

It is also not difficult
to see that $K_{I}^{\lambda,\rho}(E)/N_{I'}$ is isomorphic to
$K_{J',I'}^{\lambda',\rho'}(E)$ and, by \cite{185, DvVe3}, $K_{J',I'}^{\lambda',\rho'}(E)$ satisfies RDP$_1$.

%Now, let $\mathcal C(I)$ be the set of all connected components of $I$, and
For each $I'\in \mathcal C(I),$ let $N_{I'}$ be the normal filter defined as above.
As connected components are disjoint, we have
$\bigcap_{I'\in \mathcal C(I)} N_{I'} = \{0\}.$ This proves $K_{I}^{\lambda,\rho}(E)= K_{I,I}^{\lambda,\rho}(E)\leq \prod_{I'\in \mathcal C(I)} K_{I',J'}^{\lambda',\rho'}(E).$
\end{proof}

This result entails the following important statement on representability of kite pseudo effect algebras as a subdirect product of other subdirectly irreducible kites with RDP$_1$.

\begin{theorem}\label{th:7.5}
Every kite pseudo effect algebra $K_{I}^{\lambda,\rho}(E)$ with \RDP$_1$, where $E$ is a directed $\lambda,\rho$-weakly commutative GPEA, is  a subdirect product of subdirectly irreducible kite pseudo effect algebras with \RDP$_1.$
\end{theorem}

\begin{proof}
If $E$ is trivial, the statement is evident. Now, let $E$ be non-trivial and
take the kite $K_{I}^{\lambda,\rho}(E)$. By Theorem \ref{th:2.2}, $E$ satisfies RDP$_1.$ If the kite is not subdirectly
irreducible, then there are two possible cases:
(i) $E$ is not subdirectly irreducible, or
(ii) $E$ is subdirectly irreducible but
there exist $i,j\in I$ such that, for every
$m\in\mathbb N,$ we have $(\rho\circ\lambda^{-1})^m(i)\neq j$ and
$(\lambda\circ\rho^{-1})^m(i)\neq j$. Observe that this happens if and only if
$i$ and $j$ do not belong to the same connected component of $I$.

Now, using Remark \ref{re:2},
we can reduce (i) to (ii). So, suppose $E$ is subdirectly irreducible.
Then, using Lemma~\ref{le:7.4}, we can subdirectly embed
$K_{I}^{\lambda,\rho}(E)$ into
$\prod_{I'} K_{I',J'}^{\lambda',\rho'}(E)$, where $I'$
ranges over the connected components of $I,$ $J'=\lambda^{-1}(I')$ and $\lambda', \rho'$ are restrictions of $\lambda, \rho$ to $J'.$ But then, each
$K_{I',J'}^{\lambda,\rho}(E)$ is subdirectly irreducible by
Theorem \ref{th:6.6} and by Theorem \ref{th:2.2}, it satisfies RDP$_1.$
\end{proof}

Now let $E$ be a directed GPEA with RDP$_1.$ By Theorem \ref{th:2.2}, $E\cong \Gamma(G,G_0),$ where $G$ is a po-group with RDP$_1$ and $G_0$ is a directed GPEA that is a subalgebra of the GPEA $G^+.$ In what follows, we show a relationship between the kite pseudo effect algebras $K_{I}^{\lambda,\rho}(E)$ and $K_{I}^{\lambda,\rho}(G)_{ea}:= K_{I}^{\lambda,\rho}(G^+)$ and their relation to the po-group $G$.

\begin{theorem}\label{th:4.12}
Let $E=\Gamma(G,G_0)$ be a directed $\lambda,\rho$-weakly commutative GPEA with \RDP$_1$, where $G$ is a directed po-group satisfying and $G_0$ is a non-trivial directed GPEA which is a subalgebra of $G^+$ and $G_0$ generates $G.$
Then
\begin{enumerate}
\item[{\rm(i)}] The kite pseudo effect algebra $K_{I}^{\lambda,\rho}(E)$ is a subalgebra of the kite pseudo effect $K_{I}^{\lambda,\rho}(G^+)$.
\item[{\rm(ii)}] $K_{I}^{\lambda,\rho}(E)$ has the least non-trivial normal ideal if and only if $K_{I}^{\lambda,\rho}(G^+)$ has the least non-trivial o-ideal.
%\item[{\rm(iii)}] $K_{I}^{\lambda,\rho}(E)$ is subdirectly irreducible if and only if
\item[{\rm(iii)}] The following are equivalent:
\begin{enumerate}
\item[{\rm (1)}] $G$ has the least non-trivial o-ideal and for all $i,j\in I$ there exists an integer $m\ge 0$ such that $(\rho\circ\lambda^{-1})^m(i)=j$ or $(\lambda\circ \rho^{-1})^m(i)=j.$
\item[{\rm (2)}] $K^{\lambda,\rho}_I(E)$ has the least non-trivial normal ideal.
\end{enumerate}
%\item[{\rm(i)}]
\end{enumerate}

\end{theorem}

\begin{proof}
(i)  It is evident.

(ii) Let $J$ be the least non-trivial normal ideal of $E$ and let $N(J)$ be the normal ideal of $K^{\lambda,\rho}_I(E)$ generated by $J$. If $K$ is a non-trivial normal ideal of $K^{\lambda,\rho}_I(E)$ such that $K\subseteq N(J),$ then $K_0:=K\cap K^{\lambda,\rho}_I(E)$ is a normal ideal of $K^{\lambda,\rho}_I(E)$. Since $N(J)=\{x \in K^{\lambda,\rho}_I(G^+) \colon x \le x_1+\cdots+x_n,\ x_i \in J,\ i=1,\ldots,n,\ n \ge 1\}$, we have $K_0$ is non-trivial and whence $K_0=J$ and $N(J)$ is the least non-trivial normal ideal of $K^{\lambda,\rho}_I(G^+)$, alias $K^{\lambda,\rho}_I(G^+)$ is subdirectly irreducible.

Conversely, let $K^{\lambda,\rho}_I(G^+)$ be subdirectly irreducible with the least normal ideal $K$ and set $K_0=K\cap K^{\lambda,\rho}_I(E).$ Then $K_0$ is a normal ideal of $K^{\lambda,\rho}_I(E).$  Since every element $g\in G^+$ is the sum of elements of $E$, we see $K_0$ is a non-trivial normal ideal.  Now assume $J$ be any non-trivial normal ideal of $K^{\lambda,\rho}_I(E)$ such that $J\subseteq K_0$. We assert that $J=K_0$ otherwise the normal ideal $N(J)$ of $K^{\lambda,\rho}_I(G^+)$ generated by $J$ should be a proper subset of $J$ which is a contradiction. Hence, $K^{\lambda,\rho}_I(E)$ is subdirectly irreducible.

(iii) Since $K^{\lambda,\rho}_I(G)_{ea}= K^{\lambda,\rho}_I(G^+)$, comparing \cite[Thm 6.6]{DvuK} and Theorem \ref{th:6.6}, we see that (1) and (2) are equivalent.
\end{proof}

\begin{remark}\label{re:13}
{\rm If $E$ is trivial, the corresponding kite is a two-element Boolean algebra, therefore, Proposition \ref{pr:6.5}, Theorem \ref{th:6.7}, Lemmas \ref{le:6.8}--\ref{th:6.9}, Theorem \ref{th:4.12} are not necessarily valid.
}
\end{remark}

\section{Conclusion}

In the paper we have extended the reservoir of interesting examples of pseudo effect algebras which are closely connected with either the positive cones  po-groups as well as with generalized pseudo effect algebras and their unitization.  These pseudo algebras, called kite pseudo effect algebras, are ordinal sums of the product, $E^I$, of a generalized pseudo effect algebra $E$ with the product $(\overline E)^I$ ($\overline E$ is a copy of $E$ which is ordered in the reverse order as $E$), where operations depend on two bijections of an index set $I$.  These algebras can be both either commutative or non-commutative, and starting with a commutative generalized effect algebra, it can happen that the corresponding kite pseudo effect algebra can be non-commutative.  The resulting algebra depends on the bijections that are used, Theorem \ref{th:3.1}.

We have shown how kite pseudo effect algebras are connected with different types of the Riesz Decomposition Properties, Theorem \ref{th:3.5}. Since pseudo effect algebras are partial algebras, it is not straightforward how to study some notions of universal algebras. But for pseudo effect algebras with RDP$_1$ it is possible to study subdirect irreducibility  in an equivalent way to show when a kite pseudo effect algebra possesses the least non-trivial normal ideal, Section 4. It was shown that this property is closely connected with the existence of the least non-trivial normal ideal in the original GPEA and with special behavior of the bijections that are used. Some representation theorems, Theorem \ref{th:6.7} and Theorem \ref{th:6.9}, were presented. In addition, an open problem was formulated.

%Finally some questions are formulated.

The paper enriched the realm of pseudo effect algebras and one has again shown the importance of po-groups and generalized pseudo effect algebras for the theory of quantum structures which can be useful also for modeling events of quantum measurements when commutativity is not a priori guaranteed.

\end{document}